\documentclass[11pt, a4paper]{article}
\usepackage{amsmath}
\usepackage{amsthm}%% The amsthm package provides extended theorem environments
\usepackage{amssymb}
\usepackage{array}
\usepackage{graphicx}
\usepackage{epsfig}
\usepackage{mathrsfs}
\usepackage{multirow}
\usepackage{enumerate}
\usepackage{xcolor}

\usepackage{graphicx}
\graphicspath{%
    {converted_graphics/}% inserted by PCTeX
    {/}% inserted by PCTeX
}

\newcolumntype{"}{@{\hskip\tabcolsep\vrule width 1pt\hskip\tabcolsep}}

\def\Z{\mathbb{Z}}
\def\N{\mathbb{N}}

\definecolor{vividviolet}{rgb}{0.62, 0.0, 1.0}
\def\nt{\noindent}
\def\ms{\medskip}
\def\rsq{\hspace*{\fill}$\blacksquare$\medskip}

\newtheoremstyle{de}%name
  {10pt}          % space above
  {10pt}  % space below
  {\rm}  % bofy font
  {}%{\parindent}     % ident - empty=no indent,  \parindent= paragraph indent
  {\bf}  % thm head font
  {. }    % punctuation after thm head
  { }    % space after thm head: `` ``=normal \newline=linebreak
  {}     % thm head specification
\theoremstyle{de}

\newtheorem{de}{Definition}[section]
\newtheorem{example}{Example}[section]
\newtheorem{rem}[de]{Remark}

\newtheoremstyle{theorem}%name
  {10pt}          % space above
  {10pt}  % space below
  {\it}  % bofy font
  {}%{\parindent}     % ident - empty=no indent,  \parindent= paragraph indent
  {\bf}  % thm head font
  {. }    % punctuation after thm head
  { }    % space after thm head: `` ``=normal \newline=linebreak
  {}     % thm head specification
\theoremstyle{theorem}

\numberwithin{equation}{section}
\def\Z{\mathbb{Z}}
\def\N{\mathbb{N}}

\newtheorem{theorem}{Theorem}[section]
\newtheorem{lemma}[theorem]{Lemma}%[section]

\newtheorem{corollary}[theorem]{Corollary}%[section]
\newtheorem{conjecture}{Conjecture}[section]
\numberwithin{equation}{section}

\begin{document}
\baselineskip18truept
\normalsize
\begin{center}
{\mathversion{bold}\Large \bf On bridge graphs with local antimagic chromatic number 3}

\bigskip
{\large W.C. Shiu$^{a}$, G.C. Lau$^{b}$ and R.X. Zhang{$^c$}\footnote{Corresponding author.}}\\

\medskip

\emph{{$^a$}Department of Mathematics,\\ The Chinese University of Hong Kong,}\\
\emph{Shatin, Hong Kong.}\\
\emph{wcshiu@associate.hkbu.edu.hk}\\

\medskip

\emph{{$^b$}College of Computing, Informatics and Media,\\ Universiti Teknologi MARA,}\\
\emph{Johor Branch, Segamat Campus, 85000 Malaysia.}\\
\emph{geeclau@yahoo.com}\\

\medskip

\emph{{$^c$}School of Mathematics and Statistics, \\Qingdao University, Qingdao 266071 China.}\\
\emph{rx.zhang87@qdu.edu.cn}
\end{center}

\medskip
\begin{abstract}
Let $G=(V, E)$ be a connected graph.
A bijection $f: E\to \{1, \ldots, |E|\}$ is called a local antimagic labeling
if for any two adjacent vertices $x$ and $y$,
$f^+(x)\neq f^+(y)$, where $f^+(x)=\sum_{e\in E(x)}f(e)$
and $E(x)$ is the set of edges incident to $x$.
Thus a local antimagic labeling induces a proper vertex coloring of $G$,
where the vertex $x$ is assigned the color $f^+(x)$.
The local antimagic chromatic number $\chi_{la}(G)$ is the minimum number of colors
taken over all colorings induced by local antimagic labelings of $G$.
In this paper, we present some families of bridge graphs with $\chi_{la}(G)=3$ and give several ways to
construct bridge graphs with $\chi_{la}(G)=3$.

\medskip

\end{abstract}
%\normalsize

\tolerance=10000
%\baselineskip12truept
%\newbox\thebox
%\global\setbox\thebox=\vbox to 0.2truecm{\hsize
%0.15truecm\noindent\hfill}
%\def\boxit#1{\vbox{\hrule\hbox{\vrule\kern0pt
%     \vbox{\kern0pt#1\kern0pt}\kern0pt\vrule}\hrule}}
%\def\qed{\lower0.1cm\hbox{\noindent \boxit{\copy\thebox}}}
\def\qed{\hspace*{\fill}$\Box$\medskip}

\def\s{\,\,\,}
\def\ss{\smallskip}
\def\ms{\medskip}
\def\bs{\bigskip}
\def\c{\centerline}
\def\nt{\noindent}
\def\ul{\underline}
\def\lc{\lceil}
\def\rc{\rceil}
\def\lf{\lfloor}
\def\rf{\rfloor}
\def\a{\alpha}
\def\b{\beta}
\def\n{\nu}
\def\o{\omega}
\def\ov{\over}
\def\m{\mu}
\def\t{\tau}
\def\th{\theta}
\def\k{\kappa}
\def\l{\lambda}
\def\L{\Lambda}
\def\g{\gamma}
\def\d{\delta}
\def\D{\Delta}
\def\e{\epsilon}
\def\lg{\langle}
\def\rg{\tongle}
\def\p{\prime}
\def\sg{\sigma}
\def\to{\rightarrow}

\newcommand{\K}{K\lower0.2cm\hbox{4}\ }
\newcommand{\cl}{\centerline}
\newcommand{\om}{\omega}
\newcommand{\ben}{\begin{enumerate}}

\newcommand{\een}{\end{enumerate}}
\newcommand{\bit}{\begin{itemize}}
\newcommand{\eit}{\end{itemize}}
\newcommand{\bea}{\begin{eqnarray*}}
\newcommand{\eea}{\end{eqnarray*}}
\newcommand{\bear}{\begin{eqnarray}}
\newcommand{\eear}{\end{eqnarray}}

\section{Introduction}
For a connected graph $G=(V, E)$, a {\it local antimagic labeling} is a bijection
$f: E\to \{1, \ldots, |E|\}$ such that $f^+(u)\neq f^+(v)$ whenever the vertices $u$ and $v$ are adjacent in $G$, where $f^+(u)$ the sum of the labels of the edges incident to $u$.
Thus a local antimagic edge labeling $f$ of $G$ induces a vertex labeling $f^+: V\to \mathbb{Z}$
given by $f^+(u)=\sum_{u\in e\in E}f(e)$ such that for any two adjacent vertices with distinct labels.
This vertex labeling is called an {\it induced vertex labeling},
and the labels assigned to vertices are called {\it induced colors} under $f$ (or colors, for short, if no ambiguity occurs).
A connected graph $G$ is said to be {\it local antimagic} if it admits a local antimagic edge labeling.
Clearly, if $G$ is local antimagic, then $|V(G)|\geq 3$.
The {\it color number} of a local antimagic labeling $f$ is the number of distinct induced colors under $f$, denoted by $c(f)$, and $f$ is also called a {\it local antimagic $c(f)$-coloring}.
The {\it local antimagic chromatic number} of a graph $G$, denoted by $\chi_{la}(G)$,
is the minimum $c(f)$ over all local antimagic labelings $f$ of $G$.
Thus, $2\leq \chi_{la}(G)\leq |V(G)|$.

Local antimagic labeling was firstly proposed by Arumugam et al. \cite{Aru} in 2017.
Since then, many scholars have studied the local antimagic chromatic numbers of different types of graphs, such as Nazulaa \cite{Naz}, Lau \cite{Lau1, Lau2, LSN-IJMSI,LauShiuSoo,LSZPN}, Shaebani \cite{Sha} and so on.
Recently, Lau et al. characterized $s$-bridge graphs with local antimagic chromatic number 2 in \cite{LSZPN}.
A {\it bridge graph} (or precisely {\it $s$-bridge graph}),
denoted by $\theta(a_1,a_2,\ldots,a_s)$,
is a graph consisting of $s$ edge-disjoint $(u, v)$-paths of lengths $a_1, a_2, \ldots, a_s$,
where $s\geq 2$ and $1\leq a_1\leq a_2\leq \cdots \leq a_s$.
In the this paper, we shall further study $s$-bridge graphs $\th_s$
and present some families of $s$-bridge graphs with local antimagic chromatic number $3$.

Throughout this paper, we use $a^{[n]}$ to denote a sequence of length $n$ in which all items are $a$, where $n\ge 2$. For integers $1\le a < b$, we let $[a,b]$ denote the set of integers from $a$ to $b$. All notions and notation not defined in this paper are referred to the book \cite{Bondy}.

\section{$s$-bridge graphs with size $2s+2$}\label{sec-2s+2}
In this section, we shall study a family of $s$-bridge graphs of size $2s+2$
and show that their local antimagic chromatic number is $3$.

Firstly, we recall some known results regarding $s$-bridge graphs of size greater than $2s+2$.

\begin{theorem}[{\cite[Theorem~2.3]{LSZPN}}]\label{thm-chilath=2}  For $s\ge 3$, $\chi_{la}(\th_s) = 2$ if and only if either $\th_s=K_{2,s}$ with even $s\ge 4$ or
$\th_s$ is one of the following graphs of size greater than $2s+2$:
\begin{itemize}
\item[(1)]$\th(4l^{[3l+2]}, (4l+2)^{[l]})$, where $l\ge 1$;
\item[(2a)]$\th(2l-2, (4l-2)^{[3l-1]})$, where $l\ge 2$;
\item[(2b)]$\th(2, 4^{[3]}, 6)$; $\th(4, 8^{[5]}, 10^{[2]})$; $\th(6, 12^{[7]}, 14^{[3]})$;
\item[(3a)]$\th(4l-2-2t, 2t, (4l-4)^{[l]}, (4l-2)^{[l-2]})$, where $2\le l\le t\le \frac{5l-2}{4}$;
\item[(3b)]$\th(4l-2-2t, 2t-2, (4l-4)^{[l-1]}, (4l-2)^{[l-1]})$, where $2\le l\le t\le \frac{5l}{4}$;
\item[(4)]$\th(2t, 4s-6-2t, 2s-4, (4s-6)^{[s-3]})$, where $\frac{2s-3}{8}\le t\le \frac{6s-5}{8}$ and $s\ge 4$.
\end{itemize}

\iffalse
\begin{enumerate}[1a.]
\item[1.] $\th(4l^{[3l+2]}, (4l+2)^{[l]})$, $l\ge 1$;
\item[2a.] $\th(2l-2, (4l-2)^{[3l-1]})$, $l\ge 2$;
\item[2b.] $\th(2, 4^{[3]}, 6)$; $\th(4, 8^{[5]}, 10^{[2]})$; $\th(6, 12^{[7]}, 14^{[3]})$;
\item[3a.] $\th(4l-2-2t, 2t, (4l-4)^{[l]}, (4l-2)^{[l-2]})$, $2\le l\le t\le \frac{5l-2}{4}$;
\item[3b.] $\th(4l-2-2t, 2t-2, (4l-4)^{[l-1]}, (4l-2)^{[l-1]})$, $2\le l\le t\le \frac{5l}{4}$;
\item[4.] $\th(2t, 4s-6-2t, 2s-4, (4s-6)^{[s-3]})$, $\frac{2s-3}{8}\le t\le \frac{6s-5}{8}$ and $s\ge 4$.
\end{enumerate}
\fi
\end{theorem}

In the following of this section, we consider a family of $s$-bridge graphs $\th_{s}$ which are bipartite and of size $2s+2$, where $s\geq 3$.
Clearly, $\th_{s}=\th(2^{[s-1]}, 4)$.
In \cite{LSZPN}, it has been proved that $\chi_{la}(\th(2^{[s-1]}, 4))\geq 3$.
Next we shall show that $\chi_{la}(\th(2^{[s-1]}, 4))=3$.

\begin{theorem}\label{Thm1}
For any $s\geq 2$, $\chi_{la}(\th(2^{[s-1]}, 4))=3$.
\end{theorem}

Before proving Theorem \ref{Thm1}, we introduce some notation which will be used later.
Let $A$ be a matrix. We use $r_i(A)$ and $c_j(A)$ to denote the $i$-th row sum
and the $j$-th column sum of $A$ respectively.
And $J_{m, n}$ denotes the $m\times n$ matrix whose entries are $1$.

\begin{proof}
As $\chi_{la}(\th(2^{[s-1]}, 4))\ge 3$, it is sufficient to show that
$\chi_{la}(\th(2^{[s-1]}, 4))\le 3$ holds for any $s\geq 2$.
For any $1\leq j\leq s-1$, let $P_j=uw_jv$ be the $j$-th $(u, v)$-path in $\th_s$
and $P_s=uz_1z_2z_3v$ be the $s$-th $(u, v)$-path in $\th_s$.

\textbf{Case 1.} $s$ is odd and $s\geq 3$.

In this case, we have the following two cases.

\textbf{Case 1.1.} $s+1=4r\geq 4$.

We shall define a $2\times 4r$ matrix recurrently such that the set of its entries is $[1, 8r]$ for $r\ge 1$.

Let
$\Psi_4=\begin{pmatrix}
 1 & 7 & 6 & 3\\8 & 2 & 4 & 5\end{pmatrix}$.
It is clear that
(i) the set of entries of $\Psi_4$ is $[1,8]$,
(ii) $c_1(\Psi_4)=c_2(\Psi_4)$,
and (iii) $r_2(\Psi_4)-r_1(\Psi_4)=2$.
Suppose that $\Psi_{4r}$ is a $2\times 4r$ matrix satisfying
(i) the set of entries is $[1,8r]$,
(ii) for any $1\leq j_1, j_2\leq 4r-2$, $c_{j_1}(\Psi_{4r})=c_{j_2}(\Psi_{4r})$,
and (iii) $r_2(\Psi_{4r})-r_1(\Psi_{4r})=2$, where $r\ge 1$.

Let
 \[\Psi_{4r+4}=\left(\begin{array}{cccc|c}
 1  & 8r+7 & 8r+6 & 4& \multirow{2}{1.9cm}{$\Psi_{4r}+4J_{2,4r}$}\\
8r+8 & 2 & 3 & 8r+5 &
 \end{array}\right).\]

Consequently, we have a $2\times (s+1)$ matrix $\Psi_{s+1}=(\psi_{i,j})_{2\times (s+1)}$
satisfying the following conditions:
(i) the set of entries is $[1,2s+2]$,
(ii) for any $1\leq j_1, j_2\leq s-1$, $c_{j_1}(\Psi_{s+1})=c_{j_2}(\Psi_{s+1})$,
and (iii) $r_2(\Psi_{s+1})-r_1(\Psi_{s+1})=2$.

It is easily to get that $c_j(\Psi_{s+1})=2s+3$ for any $1\le j\le s-1$,
$r_1(\Psi_{s+1})=R-1$ and $r_2(\Psi_{s+1})=R+1$, where $R=\frac{1}{2}(s+1)(2s+3)$.
Also the last two columns of $\Psi_{s+1}$ form a matrix $\begin{pmatrix}s+3 & s\\s+1 & s+2\end{pmatrix}$.

Now we are ready to label the edges of $\th_s$.
Let $f$ be an edge labeling of $\th_s$ such that
for any $1\leq j\leq s-1$
the edges $uw_j$ and $w_jv$ are assigned the labels $\psi_{1, j}$ and $\psi_{2, j}$,
and the edges $uz_1$, $z_1z_2$, $z_2z_3$ and $z_3v$
are assigned the labels $\psi_{1,s}=s+3$, $\psi_{1,s+1}=s$, $\psi_{2, s+1}=s+2$ and $\psi_{2, s}=s+1$, respectively.
We have that
$f^+(w_j)=2s+3$ for each $1\le j \le s-1$, $f^+(z_1)=f^+(z_3)=2s+3$, $f^+(z_2)=2s+2$ and $f^+(u)=f^+(v)=R-(s+1)\ge 3s+5$.
Thus $f$ is a local antimagic $3$-coloring of $\th(2^{[s-1]}, 4)$,
which implies that $\chi_{la}(\th(2^{[s-1]}, 4))\le 3$.

\textbf{Case 1.2.} $s+1=4r+2\geq 6$.

Similar to Case~(1.1), we shall define a $2\times (4r+2)$ matrix recurrently such that the set of its entries is $[1, 8r+4]$ for $r\ge 1$.

Let $\Psi_6=\left(\begin{array}{cccc|cc}
 1 & 11 & 9 & 6 & 8 & 5\\12 & 2 & 4 & 7 & 10 & 3\end{array}\right)$.
It is obvious that for $1\leq j_1, j_2\leq 4$, $c_{j_1}(\Psi_6)=c_{j_2}(\Psi_6)=13$, and $r_1(\Psi_6)-r_2(\Psi_6)=2$.
Suppose $\Psi_{4r+2}$ is a $2\times (4r+2)$ matrix satisfying
(i) the set of entries is $[1, 8r+4]$,
(ii) for any $1\leq j_1, j_2\leq 4r$, $c_{j_1}(\Psi_{4r+2})=c_{j_2}(\Psi_{4r+2})$,
and (iii) $r_2(\Psi_{4r+2})-r_1(\Psi_{4r+2})=2$, where $r\geq 1$.

Let
 \[\Psi_{4r+6}=\left(\begin{array}{cccc|c}
 1  & 8r+11 & 8r+10 & 4& \multirow{2}{2.8cm}{$\Psi_{4r+2}+4J_{2,4r+2}$}\\
8r+12 & 2 & 3 & 8r+9 &
 \end{array}\right).\]

Clearly, $\Psi_{s+1}=(\psi_{i,j})_{2\times (s+1)}$ is a $2\times (s+1)$ matrix in which $c_j(\Psi_{s+1})=2s+3$ for each $1\le j\le s-1$ and $r_1(\Psi_{s+1})-r_2(\Psi_{s+1})=2$,
where $r_1(\Psi_{s+1})=\frac{1}{2}(2s^2+5s+5)$ and $r_2(\Psi_{s+1})=\frac{1}{2}(2s^2+5s+1)$.
Note that the last two columns of $\Psi_{s+1}$ is a matrix $\begin{pmatrix}s+3 & s\\
s+5 & s-2\end{pmatrix}$.

Let $f$ be an edge labeling of $\th_s$ such that
the edges $uw_j$ and $w_jv$ are labelled by$\psi_{1,j}$ and $\psi_{2,j}$ for each $1\le j \le s-1$,
and the edges $uz_1$, $z_1z_2$, $z_2z_3$ and $z_3v$ are labelled by $\psi_{1,s}=s+3$, $\psi_{1,s+1}=s$, $\psi_{2, s+1}=s-2$ and $\psi_{2, s}=s+5$, respectively.
Then $f^+(w_j)=2s+3$ for each $1\le j \le s-1$, $f^+(z_1)=f^+(z_3)=2s+3$, $f^+(z_2)=2s-2$, $f^+(u)=r_1(\Psi_{s+1})-s=\frac{1}{2}(2s^2+3s+5)$ and $f^+(v)=r_2(\Psi_{s+1})-(s-2)=\frac{1}{2}(2s^2+3s+5)\ge 5s+10$.
Therefore, $f$ is a local antimagic $3$-coloring of $\th(2^{[s-1]}, 4)$, implying $\chi_{la}(\th(2^{[s-1]}, 4))\le 3$.

\textbf{Case 2.} $s$ is even and $s\geq 4$.
There are also two subcases in this case.

\textbf{Case 2.1.} $s=4r\geq 4$.

Let $\Lambda_4=\left(\begin{array}{c|c}A_4 & D\end{array}\right)$, where $A_4=\begin{pmatrix}3 & 7 & 6 \\ 8 & 4 & 5\end{pmatrix}$ and  $D=\begin{pmatrix}2 \\ 1\end{pmatrix}$. Note that the set of entries of $\Lambda_4$ is $[1, 8]$ and $r_i(\Lambda_4)=18$ for each $i\in \{1,2\}$.
Suppose $\Lambda_{4r}=\left(\begin{array}{c|c}A_{4r} & D \end{array}\right)$ is defined so that the set of entries of $\Lambda_{4r}$ is $[1, 8r]$ and $r_i(\Lambda_{4r})=2r(8r+1)$ for each $i\in \{1,2\}$, where $r\ge 1$.
Let $\Lambda_{4(r+1)}=\left(\begin{array}{c|c}A_{4(r+1)} & D\end{array}\right)$, where $A_{4(r+1)}=\begin{pmatrix}A_{4r}+4J_{2,4r-1} & B_{4r}\end{pmatrix}$ and\break $B_{4r}=\begin{pmatrix} 3 & 8r+7 & 8r+6 & 6\\8r+8 & 4 & 5 & 8r+5\end{pmatrix}$.

It is clear that the set of entries of $\Lambda_{4(r+1)}$ is $[1, 8(r+1)]$ and
for each $i\in \{1,2\}$,
\begin{align*}
r_i(\Lambda_{4(r+1)}) & =r_i(A_{4(r+1)})+r_i(D)=r_i(A_{4r})+4(4r-1)+r_i(B_{4r})+r_i(D)\\
& =r_i(\Lambda_{4r})+16r-4+(16r+22)=2r(8r+1)+32r+18\\
& =2(r+1)(8r+9).
\end{align*}
Thus $\Lambda_{4n}$ is a $2\times 4n$ matrix with entry set $[1, 8n]$,
where $n\geq 1$, the sum of the elements in each row is $2n(8n+1)$,
the sum of the elements in each column except the last column is $8n+3$,
and the last column is $\begin{pmatrix}2 &1\end{pmatrix}^T$.

\textbf{Case 2.2.} $s=4r+2\geq 6$.

Let $\Lambda_6=\left(\begin{array}{c|c}A_6 & D\end{array}\right)$, where
$A_6=\left(\begin{array}{*{5}{c}}
12 & 4  & 5 & 9 & 7\\
3 & 11 & 10 & 6 & 8
\end{array}\right)$.  Note that the set of entries of $\Lambda_6$ is $[1, 12]$ and $r_i(\Lambda_6)=39$ for each $i\in \{1,2\}$.
Suppose $\Lambda_{4r+2}=\left(\begin{array}{c|c}A_{4r+2} & D \end{array}\right)$ is defined so that the set of entries of $\Lambda_{4r+2}$ is $[1, 8r+4]$ and $r_i(\Lambda_{4r+2})=2r(8r+9)+5$ for each $i\in \{1,2\}$, where $r\ge 1$.
Let $\Lambda_{4(r+1)+2}=\left(\begin{array}{c|c}A_{4(r+1)+2} & D\end{array}\right)$, where $A_{4(r+1)+2}=\begin{pmatrix}A_{4r+2}+4J_{2,4r+1} & B_{4r+2}\end{pmatrix}$ and $B_{4r+2}=\begin{pmatrix} 3 & 8r+11 & 8r+10 & 6\\8r+12 & 4 & 5 & 8r+9\end{pmatrix}$.

It is obvious that the set of entries of $\Lambda_{4(r+1)+2}$ is $[1, 8r+12]$,
and the sum of the elements in each row is
\begin{align*}
r_i(\Lambda_{4(r+1)+2}) & =r_i(A_{4(r+1)+2})+r_i(D)=r_i(A_{4r+2})+4(4r+1)+r_i(B_{4r+2})+r_i(D)\\
& =r_i(\Lambda_{4r+2})+16r+4+(16r+30)=[2r(8r+9)+5]+32r+34\\ &=2(r+1)(8r+17)+5.
\end{align*}

Therefore, $\Lambda_{4n+2}$ is a $2\times 4n+2$ matrix with entry set $[1, 8n+4]$,
where $n\geq 1$, the sum of the elements in each row is $2n(8n+9)+5$,
the sum of the elements in each column except the last column is $8n+7$,
and the last column is $\begin{pmatrix}2 &1\end{pmatrix}^T$.

Combining Case 2.1 and Case 2.2, when $s$ is even and $s\geq 4$,
we always obtain a $2\times s$ matrix $\Lambda_{s}=(\lambda_{i,j})_{2\times s}$,
in which the set of entries is $[1, 2s]$,
$r_i(\Lambda_{s})=\frac{s(2s+1)}{2}$ for each $i\in \{1,2\}$ and
$c_j(\Lambda_{s})=2s+3$ for each $1\le j\le s-1$.

For any even integer $s\geq 4$,
let $f$ be an edge labelling of $\th(2^{[s-1]}, 4)$ such that
$f(uw_j)=\lambda_{1,j}$ and $f(w_jv)=\lambda_{2,j}$ for each $1\le j\le s-1$,
$f(uz_1)=\lambda_{1,s}=2$, $f(z_1z_2)=2s+1$, $f(z_2z_3)=2s+2$ and $f(z_3v)=\lambda_{2,s}=1$.
Then we have that
$f^+(z_1)=f^+(z_3)=2s+3$,  $f^+(z_2)=4s+3$, $f^+(u)=f^+(v)=\frac{s(2s+1)}{2}=s^2+\frac{s}{2}\ge 6s+3$,
and $f^+(w_j)=2s+3$ for each $1\le j \le s-1$.
Therefore, $f$ is a local antimagic $3$-coloring,
implying that $\chi_{la}(\th(2^{[s-1]}, 4))\le 3$.

From Case 1 and Case 2, we have that $\chi_{la}(\th(2^{[s-1]}, 4))=3$ for any $s\geq 2$.
\end{proof}

\begin{rem} In Case 1.1, we get a local antimagic $3$-coloring $f$ of  $\th(2^{[s-1]}, 4)$ such that $f^+(u)=f^+(v)$.
If we require $f^+(u)\ne f^+(v)$, then there is only one case, which is $s=5$.
Following is the assignment of each $(u,v)$-path in $\th(2^{[4]}, 4)$:
$1,12; 2,11; 3,10; 4,9; 5,8,7,6$,
which induces the vertex labels $13, 15$ and $48$.
\end{rem}

\section{Bridge graphs with even number of internal paths of the same length }\label{sec-evenlength}

In this section, we shall present that there is a family of bridge graphs with even number
of internal paths of the same length whose local antimatic chromatic number is $3$.

\begin{theorem}\label{thm-even+oddlength}
For positive integers $m_1, \dots, m_l$,
$\chi_{la}(\th(m_1^{[2]},\dots, m_l^{[2]}))\le 3$,
where $l\geq 1$.
\end{theorem}

Note that $m_1, \dots, m_l$ in Theorem \ref{thm-even+oddlength} do not necessarily have to be distinct or in non-decreasing order.
Before giving the proof of Theorem \ref{thm-even+oddlength},
we firstly introduce some notion and results which will be used later.

For positive integer $s$ and integers $i, d$,
let $A_s(i;d)$ be the arithmetic progression of length $s$ with a first term $i$ and a common difference $d$.
Suppose that $A_1$ and $A_2$ are two ordered sequences of length $n$.
Let $A_1\diamond A_2$ be the ordered sequence of length $2n$ generated by
$A_1$ and $A_2$ such that the $(2i-1)$-st term is the $i$-th term of $A_1$ and the $(2i)$-th term is the $i$-th term of $A_2$, where $1\le i\le n$.
When the length of $A_1$ is $n+1$, we may similarly define $A_1\diamond A_2$ as an ordered sequence of length $2n+1$.

%In \cite{LSZPN}, the authors obtained some $s$-bridge graphs $\th_s$ with $\chi_{la}(\th_s)=2$
%by using some suitable sequences $A_1\diamond A_2$ to label the $(u, v)$-paths of $\th_s$.

When The following result can be obtained easily.

\begin{lemma}\label{lem-diamond}
 For $m,q\in\N$ and $a,b\in \Z$, let $I_{m,q}(a)=A_m(a;2)\diamond A_m(q-a; -2)$ and  $D_{m,q}(b)=A_m(b;-2)\diamond A_m(q-b+2;2)$. Then,

\begin{itemize}
\item[(1)] the first and the last terms of $I_{m,q}(a)$ are $a$ and $q-a-2m+2$;
\item[(2)] the sums of two consecutive terms of $I_{m,q}(a)$ are $q$ and $q+2$, alternatively;
\item[(3)] the first and the last terms of $D_{m,q}(b)$ are $b$ and $q-b+2m$;
\item[(4)] the sums of two consecutive terms of $D_{m,q}(b)$ are $q+2$ and $q$, alternatively.
\end{itemize}
\end{lemma}

From Lemma \ref{lem-diamond}, we obtain the following result immediately.

\begin{corollary}\label{cor-diamond}
The sum of the first term of $I_{m,q}(a)$ and that of $D_{m,q}(q+1-a)$ is $q+1$,
and the sum of the last term of $I_{m,q}(a)$ and that of $D_{m,q}(q+1-a)$ is also $q+1$.
\end{corollary}

\nt Similar to Lemma~\ref{lem-diamond} and Corollary~\ref{cor-diamond} we have
\begin{lemma}\label{lem-diamond-1}
For $m,q\in\N$ and $a,b\in \Z$, let $I^*_{m,q}(a)=A_{m+1}(a;2)\diamond A_m(q-a; -2)$ and  $D^*_{m,q}=A_{m+1}(b;-2)\diamond A_m(q-b+2;2)$. Then,

\begin{itemize}
\item[(1)] the first and the last terms of $I^*_{m,q}(a)$ are $a$ and $a+2m$;
\item[(2)] the sums of two consecutive terms of $I^*_{m,q}(a)$ are $q$ and $q+2$, alternatively;
\item[(3)] the first and the last terms of $D^*_{m,q}(b)$ are $b$ and $b-2m$;
\item[(4)] the sums of two consecutive terms of $D^*_{m,q}(b)$ are $q+2$ and $q$, alternatively;
\item[(5)] the sum of the first term of $I^*_{m,q}(a)$ and that of $D^*_{m,q}(b)$ is $a+b$,
and the sum of the last term of $I^*_{m,q}(a)$ and that of $D^*_{m,q}(b)$ is also $a+b$.
\end{itemize}
\end{lemma}

Suppose that $q$ is fixed. If there is no ambiguous, we will write $I_m(a)$,  $D_m(b)$, $I^*_m(a)$ and  $D^*_m(b)$ instead of $I_{m,q}(a)$, $D_{m,q}(b)$, $I^*_{m,q}(a)$ and $D^*_{m,q}(b)$, respectively.

Now we are ready to prove Theorem \ref{thm-even+oddlength}.

\begin{proof}[Proof of Theorem \ref{thm-even+oddlength}]
Let $\th=\th(m_1^{[2]},\dots, m_l^{[2]})$. Without loss of generality, we may assume that $m_1, \dots, m_h$ are even and $m_{h+1}, \dots, m_l$ are odd, $0\le h\le l$.
Note that if $h=0$ then there is no even length $(u,v)$-path and if $h=l$ then there is no odd length $(u,v)$-path.
For any $1\leq i\leq 2h$, let $R_i$ be a $(u,v)$-path of length $m_i=2r_i$ in $\th$,
and for any $2h+1\le i\le 2l$, let $Q_i$ be a $(u,v)$-path of length $m_i=2r_i+1$ in $\th$.
Then the size of $\th$ is $q=2\sum\limits_{i=1}^{2l} r_i +2(l-h)$.
Without loss of generality, we assume $r_1\le \cdots \le r_{2h}$ and $r_{2h+1}\le \cdots \le r_{2l}$. Thus, by the assumption, $r_{2j-1}=r_{2j}$ for $1\le j\le l$.

\nt Now, we label the edges of the $(u,v)$-paths of $\th$ by the following steps:
\begin{enumerate}[{Step }1:]
\item If $h=0$, then jump to Step~5.
\item Label the edges of $(u,v)$-path $R_{1}$ by $I_{r_1}(1)$ and the edges of $(u,v)$-path $R_{2}$ by $D_{r_2}(q)$.
\item Suppose that for each $j\geq 1$, $R_{2j}$ has been labeled.
Let $a_j$ be the least unused label. Actually $a_j=1+\sum\limits_{i=1}^{2j}r_i$. Label the edges of $(u,v)$-path $R_{2j+1}$ by $I_{r_{2j+1}}(a_j)$ and the edges of $(u,v)$-path $R_{2j+2}$ by $D_{r_{2j+2}}(q-a_j+1)$.
\item Repeat Step~3 until all $R_i$ are labeled.
\item If $h=l$, then stop;
other else, let $b_0$ be the least unused label. Actually, $b_0=1+\sum\limits_{i=1}^{2h} r_i$. Empty summation is treated as 0.
\item Label the edges of $(u,v)$-path $Q_{2h+1}$ by $I^*_{r_{2h+1}}(b_0)$ and the edges of $(u,v)$-path $Q_{2h+2}$ by $D^*_{r_{2h+2}}(q-b_0+1)$.
\item Suppose $Q_{2h+2j}$ has been labeled for each $j\ge 1$. Let $b_j$ be the least unused label. Actually $b_j=j+1+\sum\limits_{i=1}^{2h+2j}r_i$.
Label the edges of $(u,v)$-path $Q_{2h+2j+1}$ by $I^*_{r_{2h+2j+1}}(b_j)$ and the edges of $(u,v)$-path $Q_{2h+2j+2}$ by $D^*_{r_{2h+2j+2}}(q-b_j+1)$.
\item Repeat Step~7 until all $Q_i$ are labeled.
\end{enumerate}

Note that $r_{2j+1}=r_{2j+2}$. For $j\ge 0$, we have a set
\begin{align*}& \quad\ I_{r_{2j+1}}(a_j)\cup D_{r_{2j+2}}(q-a_j+1)\\ &= A_{r_{2j+1}}(a_j,2)\cup A_{r_{2j+1}}(q-a_j;-2)\cup A_{r_{2j+2}}(q-a_j+1;-2)\cup A_{r_{2j+2}}(a_j+1;2)\\
& = [a_j, a_j+2r_{2j+1}-1]\cup [q-a_j-2r_{2j+1}+2, q-a_j+1]\\
& =[a_j, a_j+r_{2j+1}+r_{2j+2}-1]\cup [q-a_j-r_{2j+1}-r_{2j+2}+2, q-a_j+1]\\
& =\left[1+\sum_{i=1}^{2j} r_i, \sum_{i=1}^{2j+2} r_i\right]\cup \left[q+1-\sum_{i=1}^{2j+2} r_i, q-\sum_{i=1}^{2j} r_i\right],\end{align*}
where $a_0=1$.
Similarly, for $1\le j\le l-h$, we have that
\begin{align*}& \quad\
 I^*_{r_{2h+2j-1}}(b_{j-1})\cup D^*_{r_{2h+2j}}(q-b_{j-1}+1)\\
 & =\left[j+\sum\limits_{i=1}^{2h+2j-2} r_i, j+\sum\limits_{i=1}^{2h+2j} r_i\right]\cup \left[q+1-j-\sum\limits_{i=1}^{2h+2j} r_i, q+1-j-\sum\limits_{i=1}^{2h+2j-2} r_i\right].
\end{align*}

Since $q=2\sum\limits_{i=1}^{2l} r_i +2(l-h)$, all labels in $[1,q]$ are assigned.
By Lemmas~\ref{lem-diamond}, \ref{lem-diamond-1} and Corollary~\ref{cor-diamond},
we can check that the induced vertex labels are $q$, $q+2$ and $l(q+1)$.
Thus the above labeling is a local antimagic $3$-coloring of $\th$,
implying that $\chi_{la}(\th)\leq 3$.
\end{proof}

Combining Theorem \ref{thm-chilath=2} and Theorem \ref{thm-even+oddlength}, we easily get the following result.

\begin{theorem} \label{th3.5}
For integers $2\le m_1\le \cdots \le m_l$, $l\ge 1$, if $\th(m_1^{[2]},\dots, m_l^{[2]})$ is not in the list of Theorem~\ref{thm-chilath=2}, then $\chi_{la}(\th(m_1^{[2]},\dots, m_l^{[2]}))= 3$.
\end{theorem}

The followings are two examples of Theorem \ref{th3.5}.

\begin{example} Consider the graph $\th(2^{[2]},4^{[2]},6^{[4]},10^{[2]})$. According to the notation used in the proof of Theorem~\ref{thm-even+oddlength}, $h=l=5$, $q=56$, $r_1=r_2=1$, $r_3=r_4=2$, $r_5=r_6=r_7=r_8=3$ and $r_9=r_{10}=5$. We label $R_i$, $1\le i\le 10$ as follows:\\
label $R_1$ by $I_{1}(1)= 1, 55$ (we shall omit the braces of the sequences in each example);\\
label $R_2$ by $D_{1}(56)= 56, 2 $;\\
label $R_3$ by $I_{2}(3)= 3, 53, 5, 51 $;\\
label $R_4$ by $D_{2}(54)= 54, 4, 52, 6 $;\\
label $R_5$ by $I_{3}(7)= 7, 49, 9, 47, 11, 45 $;\\
label $R_6$ by $D_{3}(50)= 50, 8, 48, 10, 46, 12 $;\\
label $R_7$ by $I_{3}(13)= 13, 43, 15, 41, 17, 39 $;\\
label $R_8$ by $D_{3}(44)= 44, 14, 42, 16, 40, 18 $;\\
label $R_9$ by $I_{5}(19)= 19, 37, 21, 35, 23, 33, 25, 31, 27, 29 $;\\
label $R_{10}$ by $D_{5}(38)= 38, 20, 36, 22, 34, 24, 32, 26, 30, 28 $.

\nt One may see that the induced vertex labels are 56, 58 and 285. From Theorem~\ref{thm-chilath=2} we know that $\chi_{la}(\th(2^{[2]},4^{[2]},6^{[4]},10^{[2]}))\ne 2$. Thus $\chi_{la}(\th(2^{[2]},4^{[2]},6^{[4]},10^{[2]}))=3$.
\rsq
\end{example}

\begin{example} Consider the graph $\th(2^{[2]},4^{[2]},6^{[4]},7^{[2]})$. According to the notation used in the proof of Theorem~\ref{thm-even+oddlength}, $l=5$, $h=4$, $q=50$, $r_1=r_2=1$, $r_3=r_4=2$, $r_5=r_6=r_7=r_8=3$ and $r_9=r_{10}=3$. We label $R_i$ for $1\le i\le 8$ and $Q_i$ for $i=9,10$ as follows:\\
label $R_1$ by $I_{1}(1)= 1, 49 $;\\
label $R_2$ by $D_{1}(50)= 50, 2 $;\\
label $R_3$ by $I_{2}(3)= 3, 47, 5, 45 $;\\
label $R_4$ by $D_{2}(48)= 48, 4, 46, 6 $;\\
label $R_5$ by $I_{3}(7)= 7, 43, 9, 41, 11, 39 $;\\
label $R_6$ by $D_{3}(44)= 44, 8, 42, 10, 40, 12 $;\\
label $R_7$ by $I_{3}(13)= 13, 37, 15, 35, 17, 33 $;\\
label $R_8$ by $D_{3}(38)= 38, 14, 36, 16, 34, 18 $;\\
label $Q_9$ by $I^*_{3}(19)= 19, 31, 21, 29, 23, 27, 25 $;\\
label $Q_{10}$ by $D^*_{3}(32)= 32, 20, 30, 22, 28, 24, 26 $.

\nt One may see that the induced vertex labels are 50, 52 and 255. Since $\th(2^{[2]},4^{[2]},6^{[4]}, 7^{[2]})$ is not bipartite, $\chi_{la}(\th(2^{[2]},4^{[2]},6^{[4]}, 7^{[2]}))=3$.
\rsq
\end{example}

\section{Construct some bridge graphs with local antimagic $3$-coloring}\label{sec-induced}

In the last section of this paper,
we shall give several ways to construct some bridge graphs such that
their local antimagic chromatic number is $3$.

\subsection{From bridge graphs with local antimagic 2-coloring}

From Theorem \ref{thm-chilath=2} (1),
there is a local antimagic $2$-coloring of
$\th(4l^{[3l+2]}, (4l+2)^{[l]})$, where $l\ge 1$ (See the proof of Theorem 2.3 in \cite{LSZPN}
for more details).
In what follows, we shall give a local antimagic $3$-coloring of
$\th(4l^{[3l+2]}, (4l+1)^{[l]})$ based on the local antimagic $2$-coloring of
$\th(4l^{[3l+2]}, (4l+2)^{[l]})$.

For a bridge graph $\th(4l^{[3l+2]}, (4l+2)^{[l]})$,
let $Q_j$ be the $(u, v)$-path of length $4l$, where $1\leq j\leq 3l+2$,
and $R_i$ be the $(u, v)$-path of length $4l+2$, where $1\leq i\leq l$.
Let $f$ be a local antimagic $2$-coloring of $\th(4l^{[3l+2]}, (4l+2)^{[l]})$
obtained by the method described by \cite[Theorem 2.3]{LSZPN}.
We see that the first 3 edge labels of $R_i$ starting from $u$ are $i, x-i, y-x+i$,
where $x=q+1$, $y-x=8l+4$ and $q$ is the size of $\th(4l^{[3l+2]}, (4l+2)^{[l]})$.

Now, we do the following steps to create a local antimagic $3$-coloring of $\th(4l^{[3l+2]}, (4l+1)^{[l]})$:
\begin{enumerate}[{Step }1:]
\item Delete the edge labeled by $x-i$ from $\th(4l^{[3l+2]}, (4l+2)^{[l]})$ for each $1\le i\le l$ (i.e., the second edge of $R_i$);
\item Merge two pendent vertices incident to the edges labeled by $i$ and $y-x+l-i+1$ as a new vertex $w_i$ for each $1\le i\le l$.
\end{enumerate}
Then we get the new graph $\th(4l^{[3l+2]}, (4l+1)^{[l]})$,
in which the induced vertex label of $w_i$ is $y-x+l+1=9l+5$,
and the induces vertex labels of other vertices are the same as those in $\th(4l^{[3l+2]}, (4l+2)^{[l]})$, which are $x=16l^2+10l+1$ and $y=16l^2+18l+5$.
As $\th(4l^{[3l+2]}, (4l+1)^{[l]})$ is not in the list of Theorem~\ref{thm-chilath=2},
we have the following result.
\begin{theorem} \label{th4.1}
For $l\ge 1$, $\chi_{la}(\th(4l^{[3l+2]}, (4l+1)^{[l]}))=3$.
\end{theorem}

\begin{example}\label{ex1}
We shall give a local antimagic $3$-coloring of the bridge graph $\th(8^{[8]}, 9^{[2]})$
in the following.
Firstly, we consider the bridge graph $\th(8^{[8]}, 10^{[2]})$.
Let $R_1, R_2$ be the $(u,v)$-paths of length $10$ and
$Q_i$ be the $(u,v)$-path of length $8$, where $1\le i\le 8$.
From the proof of Theorem 2.3 in \cite{LSZPN}), we\\
label $R_1$ by 1, 84, 21, 64, 41, 44, 61, 24, 81, 4;\\
label $R_{2}$ by 2, 83, 22, 63, 42, 43, 62, 23, 82, 3;\\
label $Q_1$ by 7, 78, 27, 58, 47, 28, 67, 18;\\
label $Q_2$ by 8, 77, 28, 57, 48, 37, 68, 17;\\
label $Q_3$ by 9, 76, 29, 56, 49, 36, 69, 16;\\
label $Q_4$ by 11, 74, 31, 54, 51, 34, 71, 14;\\
label $Q_5$ by 13, 72, 33, 52, 53, 32, 73, 12;\\
label $Q_6$ by 15, 70, 35, 50, 55, 30, 75, 10;\\
label $Q_7$ by 19, 66, 39, 46, 59, 26, 79, 6;\\
label $Q_8$ by 20, 65, 40, 45, 60, 25, 80, 5.

We obtain $\th(8^{[8]}, 9^{[2]})$ by deleting the edges labelled by $83$ and $84$ in $\th(8^{[8]}, 10^{[2]})$,
merging the pendent vertices incident to the edges labelled by $1$ and $83$,
and merging the pendent vertices incident to the edges labelled by $2$ and $84$.
Thus the new paths of length $9$ in $\th(8^{[8]}, 9^{[2]})$ are labelled by
$2, 21, 64, 41, 44, 61, 24, 81, 4$ and
$1, 22, 63, 42, 43, 62, 23, 82, 3$.
The labels of other paths of length $8$ are the same as those in $\th(8^{[8]}, 10^{[2]})$.
It is clear that the induced colors are $85, 105 and 23$.
Therefore, this is a local antimagic $3$-coloring for $\th(8^{[8]}, 9^{[2]})$.
\rsq\end{example}

\subsection{From some special sequences}

In what follows, we shall construct some bridge graphs of sizes $4m+3$ and $4m$,
where $m\geq 1$, from some special sequences
and show that their local antimagic chromatic numbers are $3$.

Let $A_{2m+2}(4m+3;-1)$ and $A_{2m+1}(1;1)$ be two arithmetic progressions, where $m\geq 1$.
We break the sequence $A_{2m+2}(4m+3;-1)\diamond A_{2m+1}(1;1)$ into 3 subsequences of lengths $2m+1$, $2k-1$ and $2m-2k+3$ as follows:
%{\fontsize{10}{10}\selectfont
\begin{align*}
S_1 & =A_{m+1}(4m+3;-1)\diamond A_{m}(1;1)=\{4m+3, 1,\dots, m, 3m+3\};\\
S_2 & =A_{k}(m+1;1)\diamond A_{k-1}(3m+2;-1)=\{m+1, 3m+2, \dots, 3m-k+2, m+k\},\\
S^*_3 & =A_{m-k+2}(3m-k+3;-1)\diamond A_{m-k+1}(m+k+1;1)=\{3m-k+3, m+k+1,\dots, 2m+1, 2m+2\},\end{align*}
where $1\leq k\leq m$. Note that $S_2=\{m+1\}$ when $k=1$.
let $S_3$ be the reverse of $S^*_3$, i.e., $S_3=\{2m+2,2m+1, \dots, m+k+1, 3m-k+3\}$.

Clearly, the sum of two consecutive terms of $S_i$ for each $1\le i\le 3$ is either $q+1$ or $q$.
The sum of the first terms of $S_1$, $S_2$ and $S_3$ is $(4m+3)+(m+1)+(2m+2)=7m+6$,
and that of the last terms of $S_1$, $S_2$ and $S_3$ is $(3m+3)+(m+k)+(3m-k+3)=7m+6$.

Now we label the edges of three $(u,v)$-paths of lengths $2m+1$, $2k-1$ and $2m-2k+3$ by the sequences $S_1$, $S_2$ and $S_3$, respectively.
Therefore, there is a local antimagic 3-coloring of $\th=\th(2m+1, 2k-1, 2m-2k+3)$ for each $1\le k\le m$.
Since $\th$ is not in the list of Theorem~\ref{thm-chilath=2}, we have the following result.

\begin{theorem}
For $1\le k\le m$, $\chi_{la}(\th(2k-1, 2m-2k+3, 2m+1))=3$.
\end{theorem}

Suppose that we break the above sequence $S_1$ into the following three subsequences of lengths $2l$, $2$ and $2m-2l-1$, where $1\le l\le m-1$:
\begin{align*}
T_1 & =A_{l}(4m+3;-1)\diamond A_{l}(1;1)=\{4m+3, 1,\dots, 4m-l+4, l\};\\
T_2 & =\{l+1, 4m-l+3\};\\
T_3 & =A_{m-l}(4m-l+2;-1)\diamond A_{m-l-1}(l+2;1)=\{4m-l+2, l+2,\dots, m, 3m+3\}.
\end{align*}

We label the edges of five $(u,v)$-paths of lengths $2l$, $2$, $2m-2l-1$, $2k-1$ and $2m-2k+3$ by the sequences $T_1$, $T_2$, $T_3$, $S_2$ and $S_3$, respectively.
Then there is a local antimagic 3-coloring for $\th=\th(2, 2l, 2m-2l-1, 2k-1, 2m-2k+3)$ when $1\le k\le m$ and $1\le l\le m-1$.
Also as $\th$ is not in the list of Theorem~\ref{thm-chilath=2}, we get the following result.

\begin{theorem}
For $1\le k\le m$ and $1\le l\le m-1$, $\chi_{la}(\th(2, 2l, 2m-2l-1, 2k-1, 2m-2k+3))=3$.
\end{theorem}

\begin{example} Suppose $m=3$, $k=2$. Then we have that
\begin{align*}
S_1 & = 15, 1, 14, 2, 13, 3, 12;\\
S_2 & = 4, 11, 5;\\
S_3 & = 8, 7, 9, 6, 10.
\end{align*}

Thus we label the $(u,v)$-paths of lengths $7, 3, 5$ by $S_1, S_2, S_3$, respectively.
Then we get a local antimagic $3$-coloring of $\th(3,5,7)$ with induced colors $15, 16, 27$,
implying that $\chi_{la}(\th(3,5,7))=3$.

Now we break $S_1$ into $T_1=15, 1$, $T_2= 2, 14$ and $T_3=13, 3, 12$.
Then the $(u,v)$-paths of lengths $2, 2, 3, 3, 5$ are labelled by $T_1, T_2, T_3, S_2, S_3$, respectively.
Therefore, there is a local antimagic $3$-coloring for $\th(2,2,3,3,5)$ with induced colors $15, 16, 42$.
Thus, $\chi_{la}(\th(2,2,3,3,5))=3$. \rsq
\end{example}

In what follows, we shall construct the bridge graph of size $4m$, where $m\geq 1$, from the
sequences $A_m(4m;-1)$, $A_m(1;1)$, $A_m(2m;-1)$ and $A_m(2m+1;1)$.
Let
\begin{align*}S_1 & =A_m(4m;-1)\diamond A_m(1;1)=\{4m, 1, \dots, 3m+1, m\};\\
S_2 & =A_m(2m;-1)\diamond A_m(2m+1;1)=\{2m, 2m+1, \dots, m+1, 3m\}.
\end{align*}
Note that the sums of two consecutive terms of $S_1$ (or $S_2$) are $4m+1$ and $4m$, alternatively.

We break $S_1$ into $2h$ subsequences with even lengths, i.e.,
\begin{align*}S_1^1 & =\{4m, 1, \dots, 4m+1-x_1, x_1\};\\
S_1^2 & =\{4m-x_1, x_1+1, \dots, 4m+1-x_2, x_2\};\\
\vdots & \qquad \qquad\vdots\\
S_1^{2h-1} & =\{4m-x_{2h-2}, x_{2h-2}+1, \dots, 4m+1-x_{2h-1}, x_{2h-1}\};\\
S_1^{2h} & =\{4m-x_{2h-1}, x_{2h-1}+1, \dots, 4m+1-x_{2h}, x_{2h}\}
\end{align*}

Note that $x_{2h}=m$ which is the last term of $S_1$.
That is for each $1\le i\le 2h$,
\[
S_1^i=\{4m-x_{i-1}, x_{i-1}+1, \dots, 4m+1-x_{i}, x_{i}\},
\]
where $x_0=0$ by convention. The length of $S_1^i$ is $2a_i=2(x_i-x_{i-1})$ for $1\le i\le 2h$.

Similarly, we break $S_2$ into $2k+1$ subsequences with even lengths,
i.e., for each $1\le j\le 2k+1$,
\[
S_2^j=\{4m-y_{j-1}, y_{j-1}+1, \dots, 4m+1-y_{j}, y_{j}\},
\]
where $y_0=2m$ by convention.
Note that $y_{2k+1}=3m$ and the length of $S_2^j$ is $2b_j=2(y_j-y_{j-1})$ for $1\le j\le 2k+1$.

The sums of two consecutive terms of $S_1^i$ (or $S_2^j$) are still $4m+1$ and $4m$, alternatively.
Now we reverse the order of the sequence of each $S_1^{2l}$ and keep the order of the sequence $S_1^{2l-1}$ for $1\le l\le h$.
Also reverse the order of the sequence of each $S_2^{2l}$ for $1\le l\le k$ and keep the order of the sequence $S_2^{2l-1}$ for $1\le l\le k+1$.
The new sequences are denoted by $T_1^i$ and $T_2^j$ accordingly.

The sum of all first term of $T_1^i$ is $4mh+x_{2h}=4mh+m$,
and that of all last term of $T_1^i$ is $4mh$.
Similarly, the sum of all first term of $T_2^j$ is $4m(k+1)-y_{0}=4m(k+1)-2m$,
and that of all last term of $T_2^j$ is $4mk+y_{2k+1}=4mk+3m$.
Thus the sum of all first term of $T_1^i$ and $T_2^j$ is
\[
4mh+m+4m(k+1)-2m=4m(h+k)+3m,
\]
and that of all last term of $T_1^i$ and $T_2^j$ is
\[4mh+4mk+3m=4m(h+k)+3m.\]

We label the edges of $2(h+k)+1$ $(u,v)$-paths of lengths $2a_1$, $2a_2$, $\dots$, $2a_{2h}$, $2b_1$, $2b_2$, $\dots$, $2b_{2k+1}$ by the sequences $T_1^1$, $T_1^2$, $\dots$, $T_1^{2h}$, $T_2^1$, $T_2^2$, $\dots$, $T_2^{2k+1}$, respectively.
Then the following result is obtained.

\begin{theorem}
Let $\th=\th(2a_1,2 a_2, \dots, 2a_{2h}, 2b_1, 2b_2, \dots, 2b_{2k+1})$, where $\sum\limits_{i=1}^{2h} a_i=\sum\limits_{j=1}^{2k+1} b_j$. Then $\chi_{la}(\th)=3$ if
$\th\notin \{\th(2,4^{[3]},6),\ \th(6, 12^{[7]}, 14^{[3]}),\ \th(2n, 4n-2, 6n-2, (8n-2)^{[n-2]})\}$,
where $n\ge 2$.
\end{theorem}
\begin{proof}
Let $\th=\th(2a_1,2 a_2, \dots, 2a_{2h}, 2b_1, 2b_2, \dots, 2b_{2k+1})$.
From the above discussion, $\chi_{la}(\th)\le 3$.

If $\chi_{la}(\th)=2$, then $\th$ is one of graphs listed in item $(2a)$ or $4$ of Theorem~\ref{thm-chilath=2}, since the number of internal paths is odd.

\begin{enumerate}[(a)]
\item Suppose $\th=\th(2l-2, (4l-2)^{[3l-1]})$, where $l\ge 2$.
The size of $\th$ is $(2l-2)+(3l-1)(4l-2)=4(3l^2-2l)$.

If $2l-2$ is one of $2a_i$, then all $2b_j$ are $4l-2$.

If $2l-2$ is one of $2b_j$, then all $2a_i$ are $4l-2$.

For each case, $y(2l-1)=3l^2-2l$ for some $y\in\N$. Thus $y\equiv 0\pmod l$. Let $y=kl$ for some $k\in\N$. Then $k(2l-1)=3l-2$ or $(k-1)(2l-1)=l-1$, which is no solution when $l\ge 2$.

\item Suppose $\th=\th(2t, 4s-6-2t, 2s-4, (4s-6)^{[s-3]})$, where $\frac{2s-3}{8}\le t\le \frac{6s-5}{8}$ and $s\ge 4$. Moreover, $s$ is odd, say $s=2n+1\ge 5$. The size of $\th$ is $4(s^2-3s+2)$.
Then,
    \begin{align*}ta+(2s-3-t)b+(s-2)c+(2s-3)y & =s^2-3s+2\\
    t(1-a)+(2s-3-t)(1-b)+(s-2)(1-c)+(2s-3)(s-3-y)& =s^2-3s+2
    \end{align*} for some $a,b,c\in\{0,1\}$ and $0\le y\le 2n-2$.

    Suppose $a=b=c=0$ (it is the same case when $a=b=c=1$). Then
    \[y=\frac{(2n+1)^2-3(2n+1)+2}{2(2n+1)-3}=\frac{4n^2-2n}{4n-1}=n-\frac{n}{4n-1}\notin\Z.\]

    Suppose $a=b=0$ and $c=1$ (it is the same case when $a=b=1$ and $c=0$). Then
    \[y=\frac{(2n+1)^2-4(2n+1)+4}{2(2n+1)-3}=\frac{4n^2-4n+1}{4n-1}=n-1+\frac{n}{4n-1}\notin\Z.\]

    Suppose $a=c=0$ and $b=1$ (it is the same case when $a=c=1$ and $b=0$). Then
    \[y=\frac{(2n+1)^2-5(2n+1)+5+t}{4n-1}=\frac{4n^2-6n+1+t}{4n-1}=n-2+\frac{3n-1+t}{4n-1},\] where $\frac{4n-1}{8}\le t\le \frac{12n+1}{8}$.
Thus $\frac{3n-1+t}{4n-1}\in\N$ and $0<\frac{28n-9}{32n-8}\le \frac{3n-1+t}{4n-1}\le \frac{36n-7}{32n-8}<2$, which implies that
$\frac{3n-1+t}{4n-1}=1$, i.e., $t=n$.
Therefore, $\th=\th(6n-2,(8n-2)^{[n-1]}, 2n, 4n-2, (8n-2)^{[n-1]})=\th(2n, 4n-2, 6n-2, (8n-2)^{[n-2]})$, which is excluded from the hypothesis.

     Suppose $b=c=0$ and $a=1$ (it is the same case when $b=c=1$ and $a=0$). Then
     \[y=\frac{(2n+1)^2-3(2n+1)+2-t}{4n-1}=\frac{4n^2-2n-t}{4n-1}=n-\frac{n+t}{4n-1},\]
     where $\frac{4n-1}{8}\le t\le \frac{12n+1}{8}$. Thus $\frac{n+t}{4n-1}\in\N$ and $0<\frac{12n-1}{32n-8}\le \frac{n+t}{4n-1}\le \frac{20n+1}{32n-8}<1$. There is no solution.
\end{enumerate}
\end{proof}

\begin{example} Suppose $m=6$. Then
\begin{align*}S_1 & =24, 1, 23, 2, 22, 3, 21, 4, 20, 5, 19, 6;\\
S_2 & =18, 7, 17, 8, 16, 9, 15, 10, 14, 11, 13, 12.
\end{align*}

If we choose $h=1$, $k=1$, $a_1=1$, $a_2=5$, $b_1=1$, $b_2=1$ and $b_3=4$, then we have
\begin{alignat*}{2}
T_1^1 & =24, 1; & T_2^1 & = 12, 13;\\
T_1^2 & = 6, 19, 5, 20, 4, 21, 3, 22, 2, 23; \quad & T_2^2 & = 14, 11;\\
& &T_2^3 & = 10, 15, 9, 16, 8, 17, 7, 18.
\end{alignat*}
This is a local antimagic $3$-coloring for $\th(2, 10, 2, 2, 8)=\th(2,2,2,8,10)$ with colors 24, 25, 66.

If we choose $h=1$, $k=3$, $a_1=3$, $a_2=3$, $b_1=1$, $b_2=1$, $b_3=2$, $b_4=1$ and $b_5=1$, then we have
\begin{alignat*}{2}
T_1^1 & =24, 1, 23, 2, 22, 3; & T_2^1 & = 12, 13;\\
T_1^2 & = 6, 19, 5, 20, 4, 21; \qquad & T_2^2 & = 14, 11;\\
&&T_2^3 & = 10, 15, 9, 16;\\
&&T_2^4 & = 17, 8;\\
&& T_2^5 & = 7, 18.
\end{alignat*}
This is a local antimagic $3$-coloring for $\th(6, 6, 2, 2, 4, 2, 2)=\th(2,2,2,2,4,6,6)$ with colors 24, 25, 90. \rsq
\end{example}

\subsection{From spider graphs}

A {\it spider graph} with $s\ge 2$ legs, denoted $Sp(a_1, a_2, \dots, a_s)$, is a tree formed by identifying an end-vertex, called the {\it core vertex}, of each path of length $a_i$ for each $1\le i\le s$, where $1\le a_1\le \cdots\le a_s$.
In~\cite{LauShiuSoo}, the authors obtained many sufficient conditions so that $\chi_{la}(Sp(a_1,a_2,\ldots,a_s))=s+1$. Motivated by this, we shall construct some $s$-bridge graphs from some $Sp(a_1,a_2,\ldots,a_s)$ so that their local antimagic chromatic number is 3.

\begin{theorem}[{\cite[Theorem~2.11]{LauShiuSoo}}]\label{thm-evenlegs}
Let $a_1, \ldots, a_s$ be positive even numbers, where $s\ge 2$. If $a_s = a_{s-2} + 2a_{s-3} + \cdots + (s-3)a_2 + (s-2)a_1$, then $\chi_{la}(Sp(a_1,\ldots,a_s)) = s+1$.
\end{theorem}

When $s=2$, the spider is a path. It is well-known that the local antimagic chromatic number of a path of length at least 2 is 3. Actually, the condition of Theorem~\ref{thm-evenlegs} does not hold when $s=2$.
So we only consider $s\ge 3$. From the proof of \cite[Theorem~2.11]{LauShiuSoo},
the core vertex has the induced color $q$, all vertices of degree $2$ are labelled by $q$ or $q + 1$, where $q$ is the size of the spider graph.
Now we merge all pendant vertices.
Then we get an $s$-bridge graph and a local antimagic 3-coloring for the resulting graph.
Therefore, we obtain the following result.

\begin{corollary}\label{cor-evenlegs}
Let $a_1, \ldots, a_s$ be positive even numbers, where $s\ge 3$.
If $a_s = a_{s-2} + 2a_{s-3} + \cdots + (s-3)a_2 + (s-2)a_1$,
then $\chi_{la}(\th(a_1,\ldots,a_s)) = 3$.
\end{corollary}

\begin{proof}
From the above discussion,
$\chi_{la}(\th(a_1,\ldots,a_s)) \le 3$.
It suffices to check that all graphs listed in Theorem~\ref{thm-chilath=2} do not satisfy the condition of Corollary \ref{cor-evenlegs}.

From the condition of Corollary \ref{cor-evenlegs}, we have that when $s\ge 4$,
\begin{equation*}\label{eq-condition}
a_s\ge (s-4)a_3+(s-3)a_2+(s-2)a_1.\tag{*}
\end{equation*}
Also all graphs listed in Theorem~\ref{thm-chilath=2} are $s$-bridge graphs with $s\ge 4$.
Next we show that all graphs listed in Theorem~\ref{thm-chilath=2} do not satisfy \eqref{eq-condition} respectively.
\begin{enumerate}[1a.]
\item[(1)] $\th(4l^{[3l+2]}, (4l+2)^{[l]})$, where $l\ge 1$.

Since $s=4l+2$, $a_1=4l$ and $a_{s}=4l+2$, it is clear that it does not satisfy \eqref{eq-condition}.
\item[(2a)] $\th(2l-2, (4l-2)^{[3l-1]})$, where $l\ge 2$.

Since $s=3l$, $a_1=2l-2$ and $a_{s}=4l-2$, it is clear that it does not satisfy \eqref{eq-condition}.
\item[(2b)] $\th(2, 4^{[3]}, 6)$; $\th(4, 8^{[5]}, 10^{[2]})$; $\th(6, 12^{[7]}, 14^{[3]})$.

Clearly, each graph does not satisfy \eqref{eq-condition}.
\item[(3a)] $\th(4l-2-2t, 2t, (4l-4)^{[l]}, (4l-2)^{[l-2]})$, where $2\le l\le t\le \frac{5l-2}{4}$.

Here $4\le 2l\le 2t \le \frac{5l}{2}-1$.
Thus $s=2l$, $a_s=4l-2$, $\{a_1, a_2\}=\{4l-2-2t, 2t\}$. Then,
\[
(s-3)a_2+(s-2)a_1>(s-3)(a_1+a_2)=(2l-3)(4l-2)\ge 4l-2,
\]
implying that it does not satisfy \eqref{eq-condition}.
\item[(3b)] $\th(4l-2-2t, 2t-2, (4l-4)^{[l-1]}, (4l-2)^{[l-1]})$, where $2\le l\le t\le \frac{5l}{4}$.

Similar to $(3a)$, $s=2l$, $a_s=4l-2$, $\{a_1, a_2\}=\{4l-2-2t, 2t-2\}$.
When $l\ge 3$,
\[
(s-3)a_2+(s-2)a_1>(s-3)(a_1+a_2)=(2l-3)(4l-4)\ge 4l-2.
\]
When $l=2$, $\th(4l-2-2t, 2t-2, (4l-4)^{[l-1]}, (4l-2)^{[l-1]})=\th(2,2,20,22)$.

Clearly, $a_4\ne a_2+2a_1$.
\item[(4)] $\th(2t, 4s-6-2t, 2s-4, (4s-6)^{[s-3]})$, where $\frac{2s-3}{8}\le t\le \frac{6s-5}{8}$ and $s\ge 4$.

Obviously, $\{a_1, a_2, a_3\}=\{2t, 4s-6-2t, 2s-4\}$ and $s_s=4s-6$.
When $s\ge 5$,
\[
(s-4)a_3+(s-3)a_2+(s-2)a_1>a_1+a_2+a_3=6s-10>4s-6.
\]
When $s=4$, $\th(2t, 4s-6-2t, 2s-4, (4s-6)^{[s-3]})=\th(2t, 4, 10-2t, 10)$, where $t\in \{1,2\}$.
Thus $(s-4)a_3+(s-3)a_2+(s-2)a_1=a_2+2a_1=4+4t\ne 10=a_4$.
\end{enumerate}
This completes the proof.
\end{proof}

\subsection{From one-point unions of cycles}

A {\it one-point union} of $r\ge 2$ cycles, denoted $C(n_1,n_2,\ldots,n_r)$, is a graph obtained by identifying a vertex of each cycle of order $n_1,n_2,\ldots,n_r\ge 3$.
The vertex of degree $2r$ is called the {\it core vertex} and its incident edges is called the {\it central edges}.
In~\cite[Theorem 2.4]{LSN-IJMSI}, the authors completely determined the local antimagic chromatic number of $C(n_1,n_2,\ldots,n_r)$.
 Motivated by this, we shall construct some $s$-bridge graphs from some $C(n_1,n_2,\ldots,n_r)$ so that their local antimagic chromatic number are $3$.

Note that if we choose a vertex of degree $2$ in each cycle of $C(n_1,n_2,\ldots,n_r)$ and merge these $r$ vertices into a vertex of degree $2r$ such that the graph is still simple,
then we get a $2r$-bridge graph.
For example, from $C(4,4)$, we can get $\theta(1,2,2,3)$ and $\th(2,2,2,2)$.

\begin{lemma}[{\cite[Lemma 2.3]{LSN-IJMSI}}]\label{lem-2part} Let $G$ be a graph of size $q$. Suppose there is a local antimagic labeling of $G$ inducing a $2$-coloring of $G$ with colors $x$ and $y$, where $x<y$. Let $X$ and $Y$ be the sets of vertices colored $x$ and $y$, respectively. Then $G$ is a bipartite graph with bipartition $(X,Y)$ and $|X|>|Y|$. Moreover,
$x|X|=y|Y|= \frac{q(q+1)}{2}$.
\end{lemma}

\begin{theorem}[{\cite[Theorem 2.4]{LSN-IJMSI}}]\label{thm-OUC}
For $G=C(n_1,n_2,\ldots,n_r)$, $\chi_{la}(G)=2$ if and only if $G=C((4r-2)^{[r-1]}, 2r-2)$, where $r\ge 3$,
or $G=C((2r)^{[(r-1)/2]}, (2r-2)^{[(r+1)/2]})$, where $r$ is odd.
Otherwise, $\chi_{la}(G)=3$.
\end{theorem}

\begin{rem} For completeness, we shall give the local antimagic 2-coloring $f$ for the graphs $G=C((4r-2)^{[r-1]}, 2r-2)$, where $r\ge 3$,
and $G=C((2r)^{[(r-1)/2]}, (2r-2)^{[(r+1)/2]})$, where $r$ is odd.
For $1\le i\le r$, the edges of the $i$-th cycle, $C_{n_i}$, has consecutive edges $e_{i,1}, e_{i,2}, \ldots, e_{i,n_i}$ with the central edges $e_{i,1}$ and $e_{i,n_i}$ incident to the core vertex $u$.

\ms\nt {\bf (A).} Consider $G=C((4r-2)^{[r-1]}, 2r-2)$ for $r\ge 3$.

For the $i$-th $C_{4r-2}$, we label the edges $e_{i, 2j-1}$ by $i+(2r-1)(j-1)$ and  $e_{i, 2j}$ by $4r^2-2r-i-(2r-1)j$, where $1\le j\le 2r-1$.
For $C_{2r-2}$, we label the edges $e_{r, 2j-1}$ by $(2r-1)j$ and  $e_{r, 2j}$ by $4r^2-4r+1-(2r-1)j$ for $1\le j\le r-1$.
Then the induced vertex labels are $y=4r^2-2r$ and $x=4r^2-4r+1$ alternately beginning at vertex $u$.

\ms\nt {\bf (B).} Consider  $G=C((2r)^{[(r-1)/2]}, (2r-2)^{[(r+1)/2]})$ for $r$ is odd.

For the $i$-th $C_{2r}$, where $1\le i\le (r-1)/2$, we label the edges $e_{i, 2j-1}$ by $i+2r(j-1)$ and  $e_{i, 2j}$ by $2r^2-r-i-2r(j-1)$ for $1\le j\le r$.
For the $k$-th $C_{2r-2}$, where $1\le k\le (r+1)/2$, we label the edges $e_{{k},2j-1}$ by $-r+k-1+2rj$ and $e_{{k},2j}$ by $2r^2-k+1-2rj$, $1\le j\le r-1$.
Then the induced vertex labels are $y=2r^2+r$ and $x=2r^2-r$ alternately beginning at vertex $u$. \rsq
\end{rem}

\begin{example}
For $C(10, 10, 4)$, beginning and ending with central edges, the two 10-cycles has consecutive labels 1, 24, 6, 19, 11, 14, 16, 9, 21, 4 and 2, 23, 7, 18, 12, 13, 17, 8, 22, 3 respectively while the 4-cycle has consecutive labels 5, 20, 10, 15 with $y=30$ and $x=25$.

For $C(10, 10, 8, 8, 8)$, the two 10-cycles has consecutive labels 1, 44, 11, 34, 21, 24, 31, 14, 41, 4 and 2, 43, 12, 33, 22, 23, 32, 13, 42, 3 respectively, while the three 8-cycles has consecutive labels 5, 40, 15, 30, 25, 20, 35, 10; 6, 39, 16, 29, 26, 19, 36, 9 and 7, 38, 17, 28, 27, 18, 37, 8 respectively, with $y=55$ and $x=45$. \rsq
\end{example}

Let $G=C((4r-2)^{[r-1]}, 2r-2)$, where $r\ge 3$, or $G=C((2r)^{[(r-1)/2]}, (2r-2)^{[(r+1)/2]})$, where $r$ is odd.
Let the core vertex $u$ of $G$ and all other vertices of even distant from $u$ be black vertices,
and the remaining vertices be white vertices.
For $1\le i\le r$, choose a white vertex of the $i$-th cycle of $G$, say $v_i$, such that at most one of them is adjacent to $u$ in $G$.
Identify vertices $v_1$ to $v_r$ to get a new vertex $v$.

Observe that each obtained graph is
\begin{enumerate}[(i)]
\item a $2r$-bridge graph with $u$ and $v$ the only two vertices of degree $2r$; and
\item bipartite with both partite sets of same size with $\chi_{la}\ge 3$ (by Lemma~\ref{lem-2part}).
\end{enumerate}

Note that the new graph is either $\th(2h_1-1, 4r-2h_1-1, \dots, 2h_{r-1}-1, 4r-2h_{r-1}-1, 2h_r-1, 2r-2h_r-1)$, where at most one of $2h_i-1$, $4r-2h_i-1$, $2h_r-1$, $2r-2h-1$ is $1$ and  $1\le i\le r-1$, or
$\th(2h_1-1, 2r-2h_1+1, \dots, 2h_{(r-1)/2}-1, 2r-2h_{(r-1)/2}+1, 2k_1-1, 2r-2k_1-1, \dots, 2k_{(r+1)/2}-1, 2r-2k_{(r+1)/2}-1)$, where $r$ is odd and at most one of $2h_i-1$, $2r-2h_i+1$, $2k_j-1$, $2r-2k_j-1$ is $1$, $1\le i\le (r-1)/2$ and $1\le j\le (r+1)/2$ . Clearly, these graphs are not in the list of Theorem~\ref{thm-chilath=2}.

\begin{theorem}
For simple graph $\th=\th(2h_1-1, 4r-2h_1-1, \dots, 2h_{r-1}-1, 4r-2h_{r-1}-1, 2h_r-1, 2r-2h_r-1)$ or $\th=\th(2h_1-1, 2r-2h_1+1, \dots, 2h_{(r-1)/2}-1, 2r-2h_{(r-1)/2}+1, 2k_1-1, 2r-2k_1-1, \dots, 2k_{(r+1)/2}-1, 2r-2k_{(r+1)/2}-1)$, $\chi_{la}(\th)=3$.
\end{theorem}
\begin{proof} Consider graphs obtained by the transformation given above. If $G=C((4r-2)^{[r-1]}, 2r-2)$, $r\ge 3$, we immediately have a bipartite $2r$-bridge graph that induces a local antimagic 3-coloring with vertex labels  $y=4r^2-2r$ and $x=4r^2-4r+1$ alternately beginning at vertex $u$, whereas the vertex $v$ has label $rx = r(4r^2-4r+1)$.

Similarly, if  $G=C((2r)^{[(r-1)/2]}, (2r-2)^{[(r+1)/2]})$, $r$ is odd, we immediately have a bipartite $2r$-bridge graph that induces a local antimagic 3-coloring with vertex labels  $y=2r^2+r$ and $x=2r^2-r$ alternately beginning at vertex $u$ whereas the vertex $v$ has label $rx = r^2(2r-1)$. Thus, this labeling is a local antimagic $3$-coloring.
\end{proof}

\begin{example} Suppose $G=C(10,10,4)$. Let $d(u,v_i)=d_i$, where $i\in \{1,2,3\}$,
the distance between $u$ and $v_i$. We have  $(d_1,d_2,d_3) \in \{(3,3,1), (3,5,1), (5,5,1)\}$. So the bipartite 6-bridge graphs are $\th(1,3,3,3,7,7)$, $\th(1,3,3,5,5,7)$ and $\th(1,3,5,5,5,5)$, respectively. A local antimagic 3-coloring for $\th(1,3,3,5,5,7)$ with colors 25, 30, 75 is given in Figure \ref{Fig1}.
%\centerline{\epsfig{file=th133557.eps, width=3cm}}
\end{example}
\vspace*{-0.5cm}
\begin{figure}[h]
\centering
\includegraphics[scale=0.5]{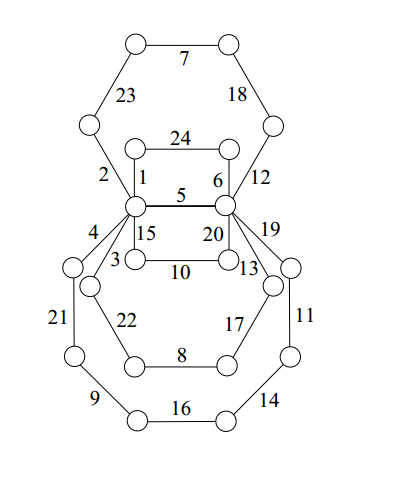}
\vspace*{-1cm}
\caption{$\th(1,3,3,5,5,7)$}\label{Fig1}
\end{figure}

Note that in transforming a $C(n_1,n_2,\ldots,n_r)$ to $2r$-bridge graphs, if at least a black vertex and at least a white vertex is chosen, the $2r$-bridge graphs must be tripartite (in this paper, tripartite graph means the chromatic number of the graph is 3).  Let the number of chosen black vertices be $k$, $1\le k\le r-1$.

\begin{theorem}   For even $s\ge 4$, there are infinitely many tripartite $s$-bridge graphs $\th$ with $\chi_{la}(\th) = 3$.  \end{theorem}

\begin{proof} Consider the tripartite $2r$-bridge graphs obtained above. If   $G=C((4r-2)^{[r-1]}, 2r-2)$, $r\ge 3$,  we immediately have a tripartite $2r$-bridge graph that induces a local antimagic 3-coloring with vertex labels  $y=4r^2-2r$ and $x=4r^2-4r+1$ alternately beginning at vertex $u$, whereas the vertex $v$ has label $ky + (r-k)x  = k(4r^2-2r) + (r-k)(4r^2-4r+1)$. Similarly, if  $G=C((2r)^{[(r-1)/2]}, (2r-2)^{[(r+1)/2]})$, $r$ is odd, we immediately have a  tripartite $2r$-bridge graph that induces a local antimagic 3-coloring with vertex labels  $y=2r^2+r$ and $x=2r^2-r$ alternately beginning at vertex $u$ whereas the vertex $v$ has label $ky + (r-k)x = k(2r^2+r) + (r-k)(2r^2-r)$. Thus, $\chi_{la}(G)\le 3$. Since $\chi(G) = 3$, the theorem holds.  \end{proof}

\begin{example} Suppose $G=C(10,10,4)$. Suppose both $d_1=d(u,v_1)$, $d_2=d(u,v_2)$ are odd, and $d_3=d(u,v_3)$ is even, i.e., $k=1$.\\ We must have $(d_1,d_2,d_3)\in \{(1,3,2), (1,5,2), (3,3,2), (3,5,2), (5,5,2)\}$. So the tripartite 6-bridge graphs are $\th(1,2,2,3,7,9)$, $\th(1,2,2,5,5,9)$, $\th(2,2,3,3,7,7)$, $\th(2,2,3,5,5,7)$, $\th(2,2,5,5,5,5)$ respectively. A local antimagic 3-coloring for $\th(2,2,5,5,5,5)$ with colors 25, 30, 80 is given in Figure \ref{Fig2.}.

%\centerline{\epsfig{file=th225555.eps, width=3.5cm}}
%\rsq
\end{example}
\vspace*{-0.5cm}
\begin{figure}[h]
\centering
\includegraphics[scale=0.5]{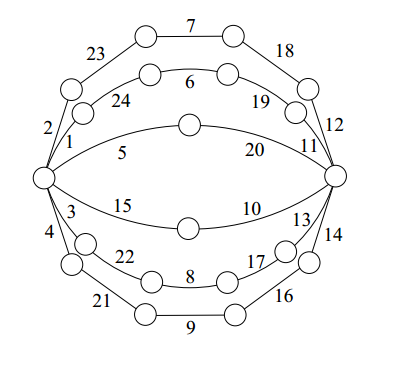}
\vspace*{-0.7cm}
\caption{$\th(2,2,5,5,5,5)$}\label{Fig2.}
\end{figure}

\begin{conjecture} If $\th$ is a bipartite bridge graph with both partite sets of same size, then $\chi_{la}(\th) = 3$. \end{conjecture}

\section*{Acknowledgement}
The third author acknowledges the support of the National Science Foundation of 
China (Grant Number 12101347) and the National Science Foundation of Shandong Province of China (Grant Number ZR2021QA085).


\begin{thebibliography}{99}

%\bibitem{Arumugam} S. Arumugam, K. Premalatha, M. Ba\v{c}a and A. Semani\v{c}ov\'{a}-Fe\v{n}ov\v{c}\'{i}kov\'{a}, Local antimagic vertex coloring of a graph, {\it Graphs and Combin.}, {\bf33} (2017)  275--285.

%\bibitem{Arumugam+L+P+W} S. Arumugam, Y.C. Lee, K. Premalatha and T.M. Wang, On local antimagic vertex coloring for corona products of graphs, (2018) arXiv:1808.04956v1.
\bibitem{Aru} S. Arumugam, K. Premalatha, M. Ba\v{c}a, A. Semani\v{c}ov\'{a}-Fe\v{n}ov\v{c}\'{i}kov\'{a}, Local antimagic vertex coloring of a graph, {\it Graphs Combin.}, {\bf 33} (2017) 275--285.

%\bibitem{Aru1} S. Arumugam, Y.C. Lee, K. Premalatha and T.M. Wang, On local antimagic vertex coloring for corona products of graphs, 2018, ArXiv:1808.04956v1.

\bibitem{Aru2} S. Arumugam, Y.C. Lee, K. Premalatha and T.M. Wang, Local antimagic chromatic number of trees-I, {\it J. Disc. Math. Sci. \& Crypto.}, {\bf 25} (2022) 1591--1602.

\bibitem{Bondy} J.A. Bondy, U.S.R. Murty, {\it Graph Theory with Applications}, New York, MacMillan, 1976.



%\bibitem{Haslegrave} J. Haslegrave, Proof of a local antimagic conjecture, {\it Discrete Math. Theor. Comput. Sci.} {\bf20}(1) (2018). https://doi.org/10.23638/DMTCS-20-1-18


\bibitem{Lau1} G.C. Lau, H.K. Ng and W.C. Shiu, Affirmative solutions on local antimagic chromatic number, {\it Graphs Combin.}, {\bf 36} (2020) 1337--1354.

\bibitem{Lau2} G.C. Lau, W.C. Shiu and H.K. Ng, On local antimagic chromatic number of cycle-related join graphs, {\it Discuss. Math. Graph Theory}, {\bf41} (2021) 133--152.

\bibitem{LSN-IJMSI}  G.C. Lau, W.C. Shiu, H.K. Ng, On local antimagic chromatic number of graphs with cut-vertices, {\it Iran. J. Math. Sci. Inform.}, (2022) arXiv:1805.04801, accepted.

%\bibitem{LauShiuNg-pendants} G.C. Lau, W.C. Shiu, H.K. Ng, On number of pendants in local antimagic chromatic number, {\it J. Disc. Math. Sci. \& Crypto.} (2021) doi: 10.1080/09720529.2021.1920190.

\bibitem{LauShiuSoo} G.C. Lau, W.C. Shiu, C.X. Soo, On local antimagic chromatic number of spider graphs, {\it J. Disc. Math. Sci. \& Crypto.} (2022) Online: https://doi.org/10.1080/09720529.2021.1892270

\bibitem{LSZPN} G.C. Lau, W.C. Shiu, R. Zhang, K. Premalatha and M. Nalliah, Complete characterization of s-bridge graphs with local antimagic chromatic number 2, manuscript, 2022.

\bibitem{Naz} N.H. Nazula, S. Slamin and D. Dafik, Local antimagic vertex coloring of unicyclic graphs, {\it Indonesian J. Combin.}, {\bf 2} (2018), 30--34.

\bibitem{Sha} S. Shaebani, On local antimagic chromatic number of graphs, {\it J. Algebr. Syst.},  {\bf 7} (2020), 245--256.

\end{thebibliography}
\end{document}